\documentclass{article}

\usepackage{graphicx}
\usepackage{subfigure}
\usepackage[T1]{fontenc}
\usepackage[latin9]{inputenc}
\usepackage{babel}
\usepackage{amsmath,amsthm,amssymb}
 
\usepackage{amsfonts}
\usepackage{authblk}
\usepackage{color}
\usepackage{enumerate}
\usepackage{natbib}
\setcitestyle{numbers,square}
\usepackage{fancyhdr}
\usepackage{float}
\usepackage{url}
\usepackage{natbib}

\usepackage{graphicx}
\usepackage{subfigure}
\usepackage[T1]{fontenc}
\usepackage[latin9]{inputenc}
\usepackage{babel}
\usepackage{amsmath}
\usepackage{amsthm}
\usepackage{amssymb}
\usepackage{amscd}
\usepackage{amsfonts}
\usepackage{enumerate}
\usepackage{fancyhdr}
\usepackage{float}
\usepackage{url}
\usepackage{natbib}
\usepackage[symbol]{footmisc}

\usepackage{latexsym}
\usepackage{verbatim}
\usepackage{bbm}
\usepackage{mathrsfs}
\usepackage{amscd}

\usepackage{cases}
\usepackage{bbm}

\newcommand{\R}{\mathbb{R}}

\providecommand{\keywords}[1]
{
\textbf{\text{Keywords: }} #1
}

\theoremstyle{plain}

\newtheorem{theorem}{Theorem}[section]
\newtheorem{proposition}[theorem]{Proposition}
\newtheorem{lemma}[theorem]{Lemma}

\theoremstyle{definition}
\newtheorem{remark}[theorem]{Remark}

\theoremstyle{definition}
\newtheorem{definition}[theorem]{Definition}

\numberwithin{equation}{section}

\begin{document}
\begin{sloppypar}	
\title{Fully nonlinear prescribed curvature problems on closed manifolds with negative curvature}
	
\author{
Jiaogen Zhang\thanks{School of Mathematics,  Hefei University of Technology,  Hefei, 230009, PR China. E-mail: zjgmath@hfut.edu.cn} 
}

\date{\today}
\maketitle

\begin{abstract}
In this manuscript, we  investigate fully nonlinear prescribed curvature problems for the modified Schouten tensor on closed Riemannian manifolds with negative curvature. We prove that whenever the corresponding concave elliptic operator satisfies a structural Condition $T$, which encompasses all $O(n)$-invariant G\r{a}rding-Dirichlet operator, such prescribed curvature problems are always solvable.
\end{abstract}


\section{Introduction}
Let $(M,g)$ be a closed (i.e. compact without boundary) connected Riemannian manifold of dimension $n \geq 3$, and let
\[
A^{t}_{g}=\frac{1}{n-2}\big( \mathrm{Ric}_{g}-\frac{t}{2(n-1)}R_{g}g\big)\,, \qquad t\in (-\infty,1)
\]
denote the modified Schouten tensor of $g$, where $\mathrm{Ric}_g$ and $R_g$ are the Ricci tensor and the scalar curvature of $g$, respectively. This family of tensors generalizes two important geometric objects: when $t=0$, is proportional to the Ricci curvature, while the case $t=1$ recovers the Schouten tensor. In conformal geometry and in the analysis of fully nonlinear elliptic equations, the tensor $A^{t}_{g}$ plays a pivotal role. More than a mere formal extension, it provides an essential framework for investigating fully nonlinear generalizations of the prescribed curvature problems.\medskip

To formulate these fully nonlinear prescribed curvature problems, we denote $\lambda(g^{-1}A_{g}^{t})$ by the eigenvalues of the $(1,1)$-tensor $g^{-1}A_{g}^{t}$. Let $\Gamma$ be a  symmetric convex cone with vertex at the origin that satisfies
\[
\Gamma_{n}=\{\lambda\in \mathbb{R}^n\, :\, \lambda_{i}>0, \, i=1,\cdots,n\}\subset \Gamma \subset \{\lambda\in \mathbb{R}^n\, :\, \sum_{i=1}^{n}\lambda_{i}>0\}=\Gamma_{1}\,,
\]
and let $f\in C^{\infty}(\Gamma)\cup C^{0}(\partial\Gamma)$ be a concave elliptic operator that satisfies the following:
\begin{itemize}\itemsep 0.5em
\item[(f1)] $f$ is strictly elliptic, i.e. $f_{i}(\lambda):=\frac{\partial f}{\partial \lambda_{i}}>0$ for every $i$.
\item[(f2)] $f$ is concave, i.e. the matrix $\big(\frac{\partial^2 f}{\partial\lambda_{i}\partial\lambda_{j}}\big)$ is negative semi-definite.
\item[(f3)]  $f>0$ in $\Gamma$ and $f=0$ on $\partial\Gamma$.
\item[(f4)] $f$ is 1-homogeneous, i.e. $f(t\lambda)=tf(\lambda)$ for every $t>0$ and for all $\lambda\in \Gamma$.
\end{itemize}\medskip

Let $[g]$ denote the conformal class of $g$, and let $\psi$ be a given smooth positive  function on $M$. The fully nonlinear prescribed curvature problem consists in finding a conformal metric $\tilde{g} \in [g]$ such that the eigenvalues of the modified Schouten tensor satisfy a prescribed fully nonlinear elliptic equation. Precisely, one can consider the following two distinct regimes:
\begin{itemize}
\item  \textit{Positive curvature regime} (namely, $\lambda(g^{-1}A^{t}_{g})\in \Gamma$):
\begin{equation}\label{positive curvature}
f(\lambda(\tilde{g}^{-1}A^{t}_{\tilde{g}}))=\psi\,, \qquad \lambda(\tilde{g}^{-1}A^{t}_{\tilde{g}})\in \Gamma\,.
\end{equation}
\item \textit{Negative curvature regime} (namely, $\lambda(-g^{-1}A^{t}_{g})\in \Gamma$):
\begin{equation}\label{negative curvature}
f(\lambda(-\tilde{g}^{-1}A^{t}_{\tilde{g}}))=\psi\,, \qquad \lambda(-\tilde{g}^{-1}A^{t}_{\tilde{g}})\in \Gamma\,.
\end{equation}
\end{itemize}\medskip

A prominent subclass of these problems concerns the elementary symmetric functions. For each $k=1,2,\cdots,n$, let $\sigma_k:\Gamma_{k}\rightarrow \mathbb{R}$ be the $k$-th elementary symmetric function, defined by
\[
\sigma_{k}(\lambda)=\sum_{1 \leq i_1 < i_2 < \cdots < i_{k} \leq n}  \lambda_{i_1} \lambda_{i_2}  \cdots \lambda_{i_k}\,,
\]
where $\Gamma_{k}=\{\lambda\in \mathbb{R}^n\, :\, \sigma_{j}(\lambda)>0, \, j=1,\cdots,k\}.$  The prescribed $\sigma_k$-curvature problem seeks to find a conformal metric $\tilde{g} \in [g]$ on $M$ solving either
\begin{equation}\label{positive curvature k}
\sigma_{k}^{1/k}(\lambda(\tilde{g}^{-1}A^{t}_{\tilde{g}}))=\psi\,, \qquad \lambda(\tilde{g}^{-1}A^{t}_{\tilde{g}})\in \Gamma_{k}\,.
\end{equation}
in the positive curvature regime, or 
\begin{equation}\label{negative curvature k}
\sigma_{k}^{1/k}(\lambda(-\tilde{g}^{-1}A^{t}_{\tilde{g}}))=\psi\,, \qquad \lambda(-\tilde{g}^{-1}A^{t}_{\tilde{g}})\in \Gamma_{k}\,.
\end{equation}
in the negative curvature regime. Since the foundational contributions of Viaclovsky \cite{Via00a} and Chang-Gursky-Yang \cite{CGY02a}, problems \eqref{positive curvature k} and \eqref{negative curvature k} have attracted sustained attention  in vast body of literature. \medskip

In the positive curvature regime, Equation \eqref{positive curvature k} has been extensively studied and resolved in a number of important settings.  Specific cases include:  $(k,n)=(2,4)$ by Chang, Gursky and Yang \cite{CGY02a,CGY02b}; locally conformally flat manifolds by Li and Li \cite{LL03} (see also Guan and Wang \cite{GW03});  $k>\frac{n}{2}$ by Gursky and Viaclovsky \cite{GV07} and $k\geq \frac{n}{2}$ by Li and Nguyen \cite{LN14} (both of which hold for any smooth positive function $\psi$); for $k=2$ and $n>8$ by Ge and Wang \cite{GW06}. Furthermore, Equation \eqref{positive curvature k} was solved via a variational approach by Brendle and Viaclovsky \cite{BV04} for the case $k = \frac{n}{2}$, and by Sheng-Trudinger-Wang \cite{STW13} for $2 \leq k \leq \frac{n}{2}$. This approach applies to the case $k=2$ and is equivalent to $M$ being locally conformally flat when $k \geq 3$ as established in \cite{BG08}. \medskip

In the negative curvature regime, Gursky and Viaclovsky \cite{GV03} established that for all $t < 1$ and $\lambda(-g^{-1}A_{g}^{t}) \in \Gamma_{k}$, there exists a unique conformal metric $\tilde{g} \in [g]$ solving  \eqref{negative curvature k}. This result was later revisited by Li and Sheng \cite{LS05}, who provided an alternative treatment using parabolic methods. More recently,  Chen-Guo-He \cite{CGH22} and Chen-He \cite{CH24} have extended this line of work to more general Krylov-type equations. It is important to emphasize that the restriction $t < 1$ is likely optimal: as noted in \cite{Via00b}, ellipticity may fail for $t > 1$, and the crucial $C^2$ estimates are unavailable for $t = 1$. \medskip

For compact manifolds with boundary ($\partial M \neq \emptyset$), A. Li and Y-Y. Li \cite{LL03b} studied the problem of finding conformal metrics of constant $\sigma_k$-scalar curvature and constant mean curvature on $\partial M$. Guan \cite{Guan08} studied the fully nonlinear equation \eqref{negative curvature} under the restriction $t \leq 0$, noting that extending the results to the full range $t < 1$ is an interesting open problem. For more related developments,  we refer the reader to Duncan-Nguyen \cite{DL25,DL26} and Duncan-Wang \cite{DW25}, and the references therein.\medskip

In contrast to prior studies, this work examines the negative curvature case specifically for closed manifolds. We contribute by: (i) establishing solvability of fully nonlinear prescribed curvature problem \eqref{negative curvature} for all $t < 1$, requiring only Condition T on the operator $(f, \Gamma)$; and (ii) developing a self-contained proof for the gradient and Hessian estimates that relies entirely on Condition T, independent of techniques like Proposition 4.4 in \cite{GV03}.
\medskip

Our main result may be stated as follows.
\begin{theorem}\label{main}
Suppose the concave elliptic operator $(f,\Gamma)$ satisfies hypotheses $(f1)-(f4)$ and Condition T. Let $(M,g)$ be a closed Riemannian manifold with $\lambda(-g^{-1}A_{g}^{t})\in \Gamma$ for some $t<1$, and let $\psi$ be a smooth positive function on $M$. Then there exists a unique conformal metric $\tilde{g}\in [g]$ satisfying
\begin{equation}\label{curvature problem}
f(\lambda(-\tilde{g}^{-1}A_{\tilde{g}}^{t}))=\psi\,, \qquad \lambda(-\tilde{g}^{-1}A_{\tilde{g}}^{t})\in \Gamma.
\end{equation}
\end{theorem}

In a recent work, Harvey and Lawson \cite{HL23} demonstrated that the class of fully nonlinear elliptic operators satisfying Condition T is unexpectedly broad, and actually includes all $O(n)$-invariant G\r{a}rding-Dirichlet operators. These operators correspond to homogeneous real polynomials $F$ defined on $\mathrm{Sym}^{2}(\mathbb{R}^n)$ with $F(I_{n})>0$, which satisfy the following conditions:
\begin{itemize}
\item[(i)] For every $A\in \mathrm{Sym}^{2}(\mathbb{R}^n)$,  the univariate polynomial 
 $t\mapsto F(tI_{n}+A)$ has only real roots;
\item[(ii)] The G\r{a}rding cone $\widetilde{{\Gamma}}$ (defined as the connected component of $\mathrm{Sym}^{2}(\mathbb{R}^n)\setminus \{F=0\}$ that contains the identity matrix $I_{n}$) includes all positive definite matrices in $\mathrm{Sym}^{2}(\mathbb{R}^n)$, i.e. $ \{A\in \mathrm{Sym}^{2}(\mathbb{R}^n): A>0\} \subset  \widetilde{{\Gamma}}$;
\item[(iii)] $F$ is invariant under the conjugation action of $O(n)$ on $\mathrm{Sym}^{2}(\mathbb{R}^n)$.
\end{itemize}
Let $F(A)=f(\lambda(A))$, where $\lambda(A)$ denotes the eigenvalue tuple of $A$; correspondingly, $\widetilde{{\Gamma}}$ can be characterized as $\{A\in \mathrm{Sym}^{2}(\mathbb{R}^n): \lambda(A)\in \Gamma\}$ for some cone $\Gamma$ in $\mathbb{R}^n$. When no confusion arises, we also refer to $f$ as an $O(n)$-invariant G\r{a}rding-Dirichlet operator on the cone $\Gamma$.\medskip

An immediate consequence of Theorem \ref{main} is therefore the following result:
\begin{theorem}\label{GD}
Let $f$ be an $O(n)$-invariant G\r{a}rding-Dirichlet operator defined on a G\r{a}rding cone $\Gamma$. Suppose $(M,g)$ is a closed Riemannian manifold with $\lambda(-g^{-1}A_{g}^{t})\in \Gamma$ for some $t<1$, and let $\psi$ be a smooth positive function on $M$. Then there exists a unique conformal metric $\tilde{g}\in [g]$ satisfying the prescribed curvature problem \eqref{curvature problem}.
\end{theorem}

In particular, setting $t = 0$  and taking  $\psi$ to be a constant, we derive the following geometric consequence:
\begin{theorem}\label{Mp Ric}
Let $f$ be an $O(n)$-invariant G\r{a}rding-Dirichlet operator defined on a cone $\Gamma$. Suppose $(M,g)$ is a closed Riemannian manifold with $\lambda(-g^{-1}\mathrm{Ric}_{g})\in \Gamma$. Then there exists a unique conformal metric $\tilde{g}\in [g]$ such that 
\[
f(\lambda(-\tilde{g}^{-1}\mathrm{Ric}_{\tilde{g}})) =constant\,, \qquad \lambda(-\tilde{g}^{-1}\mathrm{Ric}_{\tilde{g}})\in \Gamma. 
\]
\end{theorem}
This result provides a direct generalization of the classical theorem of Gursky and Viaclovsky. Indeed, if we specialize to the case $(f,\Gamma)=(\sigma_{k}^{1/k},\Gamma_{k})$, then Theorem \ref{GD} reduces precisely to \cite[Theorem 1.1]{GV03}, which states that on any closed Riemannian manifold $(M,g)$ with  $\lambda(-g^{-1}\mathrm{Ric}_{g})\in \Gamma_{k}$, there exists a unique conformal metric $\tilde{g}\in [g]$ satisfying
\[
\sigma_{k}^{1/k}(\lambda(-\tilde{g}^{-1}\mathrm{Ric}_{\tilde{g}})) =constant\,, \qquad \lambda(-\tilde{g}^{-1}\mathrm{Ric}_{\tilde{g}})\in \Gamma_{k}. 
\]

We begin by reformulating the prescribed curvature problem \eqref{curvature problem} as a fully nonlinear elliptic partial differential equation. Under the conformal change $\tilde{g} = e^{2u}g$ for an unknown function $u$ on $M$, the symmetric $(0,2)$-tensor $A_{\tilde{g}}^{t}$ could be transformed as
\[
A^{t}_{\tilde{g}}=A^{t}_{g}-\nabla^2u-\frac{1-t}{n-2}(\Delta u)g+du\otimes du-\frac{2-t}{2}|\nabla u|^2g\,.
\]
For any symmetric $(0,2)$-tensor $W$, we introduce the operator notation
\begin{equation*}
F(W)=f(\lambda(g^{-1}W))\,.
\end{equation*}
Consequently,  Equation \eqref{curvature problem} can be expressed in terms of the background metric $g$ as the following scalar equation:
\begin{equation}\label{scalar curvature problem}
F\Big(\nabla^2u+\frac{1-t}{n-2}(\Delta u)g+\frac{2-t}{2}|\nabla u|^2g-du\otimes du-A^{t}_{g}\Big)=\psi e^{2u}\,.
\end{equation}
Thus, solving the prescribed curvature problem \eqref{curvature problem} is reduced to establishing uniform a priori estimates for solutions $u$ to the second-order  fully nonlinear elliptic equation \eqref{scalar curvature problem}.\medskip

We first observe that Equation \eqref{scalar curvature problem} falls within the general class of equations given by
\begin{equation}\label{extension}
F(\nabla^2u+S(x,\nabla u))=\psi(x,u)\,, \qquad x\in M\,,
\end{equation}
where $S = S(\cdot,\nabla u)$ denotes a symmetric $(0,2)$-tensor that also depends on $\nabla u$. In their seminal work \cite{CNS85}, Caffarelli-Nirenberg-Spruck investigated the Dirichlet problem for Equation \eqref{extension} in Euclidean space. Their research has been fundamentally influential in advancing the theory and applications of fully nonlinear elliptic and parabolic equations. Since then, numerous authors have made important contributions to this subject from diverse perspectives, including Chou-Wang \cite{CW}, Dong \cite{Dong06}, Guan \cite{Guan94,Guan99,Guan12}, Guan and Li \cite{GL96}, Harvey-Lawson \cite{HL09,HL11}, Ivochkina \cite{Iv91}, Ivochkina-Trudinger-Wang \cite{ITW04},  Krylov \cite{Kry83,Kry95,Kry97}, Li \cite{Li90}, Trudinger \cite{Tru95}, Trudinger-Wang \cite{TW99}, Urbas \cite{Urbas}, Wang \cite{W94}, Yuan \cite{Yuan01}, and many others. \medskip

The organization of this paper is as follows: In Section 2, we start by introducing the terminology of Condition $T$ and present several well-known examples of elliptic operators that satisfy this condition. We then analyze the ellipticity of the prescribed curvature problem \eqref{curvature problem}--a result that will be crucial for subsequent arguments.

The core argument in this work is Section 3, which relies on an intricate application of the maximum principle, first to obtain a $C^0$ estimate for the solution, then a gradient bound, and finally the full $C^2$ estimate by constructing appropriate auxiliary functions and controlling the relevant tensor terms.

With these uniform estimates established, we proceed in Section \ref{proof} to prove the main existence theorem. This is achieved via the continuity method. The a priori estimates are essential for the closedness part, while the ellipticity result is used in the openness argument via the invertibility of the linearized operator. 

In the final section, we present an application of the main theorem to the prescribed $\mathcal{M}_p$-curvature problem in the negative case.
\medskip

\noindent\textbf{Acknowledgements}
The author was supported by the National Natural Science Foundation of China (Grant No. 12401065) and the Initial Scientific Research Fund of Young Teachers in Hefei University of Technology (Grant No. JZ2024HGQA0119). The author is deeply grateful to his thesis advisor Prof. Xi Zhang for his constantly guidance and encouragement. He also wishes to thank Prof. Li Chen for helpful comments and suggestions. 
\medskip

\section{Preliminaries}

\subsection{Condition T}
Under a suitable structural hypothesis, the theory developed for Monge-Amp\`ere equations can often be extended to much broader classes of fully nonlinear concave elliptic equations. The condition introduced below provides one natural candidate that applies in many important settings.
\begin{definition}\label{T}
We say a concave elliptic operator $f: \Gamma \rightarrow \mathbb{R}$ satisfies the Condition T if there exists a uniform  constant  $\mathscr{T}>0$ such that  
\begin{equation}\label{det}
\prod_{i=1}^{n}\frac{\partial f}{\partial \lambda_{i}}\geq  {\mathscr{T}}\,.
\end{equation}  
\end{definition}
Heuristically, Condition $T$ enables us to exploit a known upper bound for $\{\frac{\partial f}{\partial \lambda_{i}}\}_{i=1}^{n}$  to derive a corresponding lower bound, thereby achieving two-sided control. Such arguments are common in establishing a priori estimates, particularly when the maximum principle is applied to the linearized operator. We also note that this hypothesis has been employed in the recent work of Guo and Phong \cite{GP24}, where they investigate $L^{\infty}$ estimates for fully nonlinear equations on non-K\"ahler manifolds. \medskip

For an arbitrary $O(n)$-invariant G\r{a}rding-Dirichlet operator $f$ on the G\r{a}rding cone $\Gamma$, Harvey and Lawson \cite{HL23} established the following strict ellipticity property:
\begin{equation}\label{G-D}
f(\lambda+\tau)\geq f(\lambda)+\Big(\prod_{i=1}^{n}  \tau_i\Big)^{\frac{1}{n}} \qquad \text{for all } \lambda\in \Gamma\, \text{ and } \tau\in \Gamma_{n}\,.
\end{equation}
It can be shown that inequalities \eqref{det} and \eqref{G-D} are equivalent in a certain sense; see, for example, \cite[Lemma 2.4]{Zhang25}.  Consequently, Condition T is satisfied by every $O(n)$-invariant G\r{a}rding-Dirichlet operator. 

\begin{remark}\label{example}
\begin{itemize}\itemsep 0.5em
\item[(1)] Examples of $(f,\Gamma)$ satisfying Condition  T including 
\begin{itemize}
\item[(i)] Hessian operators $(\sigma_{k}^{1/k},\Gamma_{k})$ for $1\leq k\leq n$. The verification of Condition T can be found in the work of Wang \cite{W94}.
\item[(ii)] $p$-Monge-Amp\`ere operator $(\mathcal{M}^{1/\binom{n}{p}}_{p},\mathcal{P}_{p})$ for $1\leq p\leq n$, where $\mathcal{M}_{p}$ is a symmetric polynomial given by
\begin{equation}
\mathcal{M}_{p}(\lambda)=\prod_{1\leq i_1<i_2<\cdots <i_{p}\leq n} (\lambda_{i_1}+\lambda_{i_2}+\cdots +\lambda_{i_p})\,,
\end{equation}
where
\begin{equation}
\mathcal{P}_{p}=\{ \lambda\in \R^n \,: \,\lambda_{i_1}+\cdots +\lambda_{i_{p}}>0,\, 1\leq i_1<\cdots< i_{p}\leq n\}
\end{equation}
The proof of Condition T has been provided in \cite{D21,ZZ25}. 
\end{itemize} 
\item[(2)] Notably for all $n\geq k>l\geq 1$, the Hessian quotient equations  $((\frac{\sigma_{k}}{\sigma_{l}})^{1/(k-l)}, \Gamma_{k})$  fail to meet the  Condition T.
\end{itemize}
\end{remark}

\subsection{Ellipticity of the prescribed curvature problems}
Let us define an endomorphism
\begin{equation}\label{Wu}
W[u]:=\nabla^2u+\frac{1-t}{n-2}(\Delta u)g+\frac{2-t}{2}|\nabla u|^2g-du\otimes du-A^{t}_{g}\,.
\end{equation}
Then we can reformulate \eqref{scalar curvature problem} as the following simpler form
\begin{equation}\label{scalar two}
F\big(W[u]\big) =\psi e^{2u}\,.
\end{equation}
\begin{lemma}\label{elliptic 1}
Suppose $\lambda(-g^{-1}A_{g}^{t})\in \Gamma$ for some $t\leq 1$, then for any solution $u$ of \eqref{scalar two}, we have $W[u] \in \Gamma$.
\end{lemma}
\begin{proof}
At a minimum point $x_0 \in M$ of $u$, we have $\nabla u(x_0)=0$ and $\nabla^2 u(x_0) \geq 0$. Using the hypothesis $\lambda(-g^{-1}A_{g}^{t})\in \Gamma$ and the expression for $W[u]$, we obtain
\[
W[u](x_0)=\nabla^2u(x_0)+\frac{1-t}{n-2}(\Delta u(x_0))g-g^{-1}A_{g}^{t}(x_0)\geq -g^{-1}A_{g}^{t}(x_0)\,. 
\]  
This implies $W[u](x_0) \in \Gamma$ as we have assumed $\lambda(-g^{-1}A_g^t(x_0)) \in \Gamma$. By the connectedness of $\Gamma$, the result $W[u] \in \Gamma$ follows.
\end{proof}

\begin{lemma}\label{elliptic}
Suppose $\lambda(-g^{-1}A_{g}^{t})\in \Gamma$ for some $t\leq 1$, then Equation \eqref{scalar two} is elliptic.
\end{lemma}
\begin{proof}
Given a smooth function $v$, define a one-parameter family $u_{s}=u+s v$. Let us consider
\[
G[u_s,\nabla u_s,\nabla^2u_s]=F\big(W[u_{s}]\big)-\psi e^{2u_{s}}\,.
\]
For any symmetric matrix $U=\{U_{ij}\}$ with $\lambda(g^{-1}U)\in \Gamma$,  the linearized operator is represented by the positive-definite matrix
\[
\big\{F^{ij}(U)\big\}=\bigg\{\frac{\partial F}{\partial U_{ij}}\bigg\}\,.
\]
A direct computation yields
\[
\frac{d}{ds}W[u_{s}]\mid_{s=0}=\nabla^2v+\frac{1-t}{n-2}(\Delta v)g+\chi(dv)\,,  
\]
with $\chi(dv)=(2-t)\langle du,dv\rangle g -2du\otimes dv$ being linear in $dv$. It follows that
\[
\begin{split}
&\frac{d}{ds}F\big(W[u_{s}]\big)\mid_{s=0}\\
&=F^{ij}\big(W[u]\big)\Big(\nabla_{i}\nabla_{j}v+\frac{1-t}{n-2}(\Delta v)g_{ij}+\chi_{ij}(dv) \Big)\\
&=\Big(F^{ij}\big(W[u]\big)+\frac{1-t}{n-2}F^{kl}\big(W[u]\big)g_{kl}g^{ij}\Big)\nabla_{i}\nabla_{j}v+F^{ij}\big(W[u]\big)\chi_{ij}(dv)\,.
\end{split}
\]
By Lemma \ref{elliptic 1}, we know that the following matrix 
\begin{equation}\label{cofactor}
\big\{\overline{F}^{ij}(W[u])\big\}=\Big\{F^{ij}\big(W[u]\big)+\frac{1-t}{n-2}F^{kl}\big(W[u]\big)g_{kl} g^{ij}\Big\}
\end{equation}
is positive definite. Then we conclude that the operator
\[
\begin{split}
L(v):&=\frac{d}{ds}G[u_s,\nabla u_s,\nabla^2u_s]\mid_{s=0}\\
&=\overline{F}^{ij}(W[u])\nabla_{i}\nabla_{j}v+F^{ij}\big(W[u]\big)\chi_{ij}(dv)-\psi e^{2u}v
\end{split}
\]
is indeed elliptic.
\end{proof}

\section{A priori estimates}

\subsection{Zeroth order estimate}
In this subsection, we prove the following $C^{0}$ estimate:
\begin{proposition}
Assume $\lambda(-g^{-1}A_{g}^{t})\in \Gamma$ for some $t\leq 1$. Then there exists a constant $C_{0}$ depending only on $\psi$, $A_{g}^{t}$, and $n$, such that for any solution $u$ of \eqref{scalar curvature problem}, we have
\begin{equation*}
\sup_{M}|u|\leq C_{0}\,.
\end{equation*}
\end{proposition}
\begin{proof}
We will prove this result using the standard maximum principle. Since $M$ is compact,   assume that $u$ attain its minimum at a point $x_{min}\in M$. At this point, we have $\nabla u(x_{min})=0$ and $\nabla^2 u(x_{min})$ is positive semi-definite. Consequently
\[\nabla^2u+\frac{1-t}{n-2}(\Delta u)g\geq 0\]
at $x_{min}$, which implies
\[
F(-A^{t}_{g}(x_{min}))\leq \psi(x_{min})e^{2u(x_{min})}\,.
\]
It follows that
\[
u(x_{min})\geq \frac{1}{2}\log \bigg(\frac{F(-A^{t}_{g}(x_{min}))}{\psi(x_{min})} \bigg)\geq \frac{1}{2}\min_{M}\log \bigg(\frac{F(-A^{t}_{g})}{\psi} \bigg)\,.
\]
Similarly, if $u$ attain its maximum at a point $x_{max}\in M$, then
\[
u(x_{max})\leq \frac{1}{2}\log \bigg(\frac{F(-A^{t}_{g}(x_{max}))}{\psi(x_{max})} \bigg)\geq \frac{1}{2}\max_{M}\log \bigg(\frac{F(-A^{t}_{g})}{\psi} \bigg)\,.
\]
This completes the proof.
\end{proof}

\subsection{First order estimate}

In this subsection, we shall concentrate on the following global gradient estimate on  \eqref{scalar curvature problem}.
\begin{proposition}\label{gradient estimate}
Assume  $\lambda(-g^{-1}A_{g}^{t})\in \Gamma$ for some $t\leq 1$. Let $u\in C^{3}(M)$ be a solution to  \eqref{scalar curvature problem}. Then we have
\begin{equation*}
\sup_{M}|\nabla u|\leq C_{1}\,,
\end{equation*}
where $C_{1}$ is a constant depends only on $\psi$, $A_{g}^{t}$, $n$, and   $C_{0}=\|u\|_{C^{0}}$.
\end{proposition}

\begin{proof}
To proceed, we introduce an auxiliary function following a strategy of Gursky and Viaclovsky \cite{GV03}, which itself builds upon earlier work by Li 
 \cite{Li90}. Compared to the auxiliary function in \cite{GV03}, our construction is much simpler, which is defined as follows:
\begin{equation*}
G=\log\beta +\phi(u) \qquad  \textrm{ with }  \beta=1+\frac{1}{2}|\nabla u|^2, 
\end{equation*}
where $\phi$ is a monotone increasing function  given by
\[
\phi(s)=-\log (e^{2+2C_{0}}-e^{2s})\,, \quad s\in [-C_{0},C_{0}].
\]
By compactness, we may assume that the maximum of $G$ over $M$ is attained at some point $x_0 \in M$. Near this maximum point, we work in a smooth orthonormal local frame $\{e_1,\cdots,e_n\}$ in the vicinity of the point $x_0$ such that the following statements hold:
\begin{itemize}\itemsep 0.4em
\item the Christoffel symbols vanishes at $x_0$, i.e. $\Gamma_{ij}^{k}=0$ ;
\item the metric matrix $\big\{g_{ij}\big\}$ is diagonal at $x_0$; 
\item  at $x_0$, the identity
\begin{equation}\label{normal coodinates}
\sum_{k,l=1}^{n}\bigg(\nabla_{i}\nabla_{j}(g^{kl})+2\nabla_{l}(\Gamma_{ij}^{k})\bigg)u_{k}u_{l}=2\sum_{k,l=1}^{n}R_{iljk}u_{k}u_{l} 
\end{equation}
holds, where $R_{iljk}$ denotes the Riemann curvature tensor components of the metric $g$.
\end{itemize}
We refer to \cite{Via02} for a detailed proof.

For notational convenience, we denote
\[
\mathcal{F}=\sum_{k=1}^{n}F^{kk}\,.
\]
It follows from \eqref{cofactor} that
\begin{equation}\label{relation}
\overline{F}^{ij}=F^{ij}+\frac{1-t}{n-2}\mathcal{F}\delta_{ij}.
\end{equation}
Furthermore, we set
\[
L=\overline{F}^{ij}e_{i}e_{j}=F^{ij}e_{i}e_{j}+\frac{1-t}{n-2}\mathcal{F}\Delta.
\]
The condition that $G$ is maximized at $x_0$ yields
\begin{equation}\label{vanishes1}
0=\nabla_{i}G=\frac{1}{\beta}u_{k}u_{ki}+ \phi'u_{i}\,, \quad i\in \{1,\cdots,n\}\,,
\end{equation}
hence the Hessian matrix satisfies
\begin{equation*}
\begin{split}
\nabla_{ij}G=& \frac{1}{\beta}\big(u_{k}u_{kij}+u_{ki}u_{kj}+\frac{1}{2}\,(g^{kl})_{ij}u_{k}u_{l}\big)-\frac{1}{\beta^2}(u_{k}u_{ki})(u_{l}u_{lj})+\phi''u_{i}u_{j}+\phi'u_{ij} \\
=& \frac{1}{\beta}\big(u_{k}u_{kij}+u_{ki}u_{kj}+\frac{1}{2}\,(g^{kl})_{ij}u_{k}u_{l}\big)+(\phi''-\phi'^2)u_{i}u_{j}+\phi'u_{ij}\,,
\end{split}
\end{equation*}
which is negative semi-definite. This yields the inequality
\begin{equation}\label{maximum principle 1}
\begin{split}
0\geq \frac{1}{\beta}\overline{F}^{ij}u_{l}u_{lij}+\frac{1}{2\beta}L(g^{kl})u_{k}u_{l}+(\phi''-\phi'^2)\overline{F}^{ij}u_{i}u_{j}+\phi'L(u).
\end{split}
\end{equation}

We first estimate the last term in \eqref{maximum principle 1}. In the local coordinates above, we have
\[
(W[u])_{ij}=u_{ij}+\frac{1-t}{n-2}(\Delta u)\delta_{ij}+\frac{2-t}{2}|\nabla u|^2\delta_{ij}-u_{i}u_{j}-A_{ij}\,,
\]
where $A^{t}_{ij}=(A^{t}_{g})_{ij}$. Therefore, 
\begin{equation}\label{Lu}
\begin{split}
L(u)=&F^{ij}u_{ij}+\frac{1-t}{n-2}\mathcal{F}\Delta u\\
=&F^{ij}\big((W[u])_{ij}-\frac{1-t}{n-2}(\Delta u)\delta_{ij}-\frac{2-t}{2}|\nabla u|^2\delta_{ij}+u_{i}u_{j}+A^{t}_{ij} \big)+\frac{1-t}{n-2}\mathcal{F}\Delta u\\
=&\psi e^{2u}-\frac{2-t}{2}|\nabla u|^2\mathcal{F}+F^{ij}\big(u_{i}u_{j}+A^{t}_{ij} \big).
\end{split}
\end{equation}

Next, we estimate the first term in \eqref{maximum principle 1}. Differentiating the equation \eqref{scalar two} and using $\Gamma_{ij}^{k}=0$ at $x_0$, we deduce from \eqref{vanishes1} that
\begin{equation*}
\begin{split}
(\psi_{l}+2\psi u_{l})e^{2u}&=F^{ij}\big[(W[u])_{ij}\big]_{l}\\
&=F^{ij}\big[u_{ij}+\frac{1-t}{n-2}(\Delta u)g_{ij}+\frac{2-t}{2}|\nabla u|^2g_{ij}-u_{i}u_{j}-A^{t}_{ij}\big]_{l}\\
&=F^{ij}\big(u_{ijl}-u_{k}(\Gamma_{ij}^{m})_{l} \big)-F^{ij}\big(u_{il}u_{j}+u_{i}u_{jl}+\nabla_{l}A^{t}_{ij}\big)\\
& \quad +\Big(\frac{1-t}{n-2}(u_{kkl}-u_{i}(\Gamma_{kk}^{i})_{l})-(2-t)\beta\phi' u_{l}\Big)\mathcal{F}\,.
\end{split}
\end{equation*}
Note that \eqref{relation} implies
\[
\overline{F}^{ij}u_{ijl}=F^{ij}u_{ijl}+\frac{1-t}{n-2}\mathcal{F}u_{kkl}.
\]
It follows that
\begin{equation*}
\begin{split}
(\psi_{l}+2\psi u_{l})e^{2u}
=&\overline{F}^{ij}u_{ijl}-\Big((2-t)\beta\phi' u_{l}+\frac{1-t}{n-2}u_{i}(\Gamma_{kk}^{i})_{l}\Big)\mathcal{F}\\
&\qquad -F^{ij}\big(u_{il}u_{j}+u_{i}u_{jl}+\nabla_{l}A^{t}_{ij}+u_{k}(\Gamma_{ij}^{k})_{l}\big)\,.
\end{split}
\end{equation*}
Multiplying by $u_{l}$ and using \eqref{vanishes1} again, we obtain
\begin{equation}\label{13terms}
\begin{split}
&\overline{F}^{ij}u_{l}u_{lij}\\
&=(\langle \nabla \psi, \nabla u\rangle +2\psi|\nabla u|^2)e^{2u}+(2-t)\beta\phi' |\nabla u|^2\mathcal{F}+\frac{1-t}{n-2}u_{i}u_{l}(\Gamma_{kk}^{i})_{l} \mathcal{F}\\
&\quad +F^{ij}\big(u_{il}u_{j}u_{l}+u_{i}u_{l}u_{jl}+u_{l}\nabla_{l}A^{t}_{ij}+u_{k}u_{l}(\Gamma_{ij}^{k})_{l}\big)\\
&=(\langle \nabla \psi, \nabla u\rangle +2\psi|\nabla u|^2)e^{2u}+ (2-t)\beta\phi' |\nabla u|^2\mathcal{F}+\frac{1-t}{n-2}u_{i}u_{l}(\Gamma_{kk}^{i})_{l} \mathcal{F}\\
&\quad +F^{ij}\big(u_{l}\nabla_{l}A^{t}_{ij}+u_{k}u_{l}(\Gamma_{ij}^{k})_{l}\big)-2\beta\phi' F^{ij}u_{i}u_{j}\,.
\end{split}
\end{equation}

Substituting \eqref{Lu} and \eqref{13terms} into \eqref{maximum principle 1}, we infer  that
\begin{equation}\label{inserting}
\begin{split}
0\geq &\,  \frac{1}{\beta}(\langle \nabla \psi, \nabla u\rangle +2\psi|\nabla u|^2)e^{2u}+\phi'\psi e^{2u}+  \frac{2-t}{2}\phi' |\nabla u|^2\mathcal{F}\\
&+\frac{1}{\beta}\frac{1-t}{n-2}u_{i}u_{l}(\Gamma_{kk}^{i})_{l} \mathcal{F}  + \frac{1}{\beta}F^{ij}\big(u_{l}\nabla_{l}A^{t}_{ij}+u_{k}u_{l}(\Gamma_{ij}^{k})_{l}\big)\\
&+(\phi''-\phi'^2-\phi')F^{ij}u_{i}u_{j}+\frac{1}{2\beta}F^{ij}e_{i}e_{j}(g^{kl})u_{k}u_{l}\\
&+\frac{1}{2\beta}\frac{1-t}{n-2}\Delta(g^{kl})u_{k}u_{l}\mathcal{F}+\frac{1-t}{n-2}(\phi''-\phi'^2)|\nabla u|^2\mathcal{F}+\phi'F^{ij}A^{t}_{ij}.
\end{split}
\end{equation}
In view of \eqref{normal coodinates}, we obtain
\[
\frac{1}{\beta}F^{ij}u_{k}u_{l}(\Gamma_{ij}^{k})_{l}+\frac{1}{2\beta}F^{ij}e_{i}e_{j}(g^{kl})u_{k}u_{l}=\frac{1}{\beta}F^{ij}R_{iljk}u_{k}u_{l}
\]
and
\[
\frac{1}{\beta}\frac{1-t}{n-2}u_{i}u_{l}(\Gamma_{kk}^{i})_{l} \mathcal{F}+\frac{1}{2\beta}\frac{1-t}{n-2}\Delta(g^{kl})u_{k}u_{l}\mathcal{F}=\frac{1}{\beta}\frac{1-t}{n-2}R_{mlmk}u_{k}u_{l}\mathcal{F}\,.
\]
Plugging these equalities into \eqref{inserting}, we deduce that
\begin{equation}\label{inserting 2}
\begin{split}
0\geq & \,  \frac{1}{\beta}(\langle \nabla \psi, \nabla u\rangle +2\psi|\nabla u|^2+\beta \phi'\psi)e^{2u}+\frac{1}{\beta}\frac{1-t}{n-2}R_{mlmk}u_{k}u_{l}\mathcal{F} \\
&+ \frac{1}{\beta}F^{ij} u_{l}\nabla_{l}A^{t}_{ij} +(\phi''-\phi'^2-\phi')F^{ij}u_{i}u_{j}+\frac{1}{\beta}F^{ij}R_{iljk}u_{k}u_{l}\\
&+\Big(\frac{1-t}{n-2}(\phi''-\phi'^2)+\frac{2-t}{2}\phi'\Big)|\nabla u|^2\mathcal{F}+\phi'F^{ij}A^{t}_{ij}.
\end{split}
\end{equation}

From the definition of $\phi$, one may verify that
\[
\begin{split}
\phi'(s)&=\frac{2e^{2s}}{e^{2+2C_0}-e^{2s}}=\frac{2}{e^{2+2C_{0}-2s}-1}\,,\\ 
\phi''(s)&=\frac{4e^{2+2C_{0}+2s}}{(e^{2+2C_0}-e^{2s})^{2}}\,,
\end{split}
\]
which guarantees
\begin{equation*} 
 \phi''-\phi'^2-\phi'=\phi' \quad  \textrm{and}
\quad  \phi'>\epsilon_{0}:=\frac{2}{e^{2+4C_{0}}-1}\,.
\end{equation*}
It follows from \eqref{inserting 2} that
\begin{equation}\label{inserting 3}
\begin{split}
0\geq & \, \epsilon_{0}F^{ij}u_{i}u_{j}+F^{ij}U_{ij}+\frac{1}{\beta}\big(\langle \nabla \psi, \nabla u\rangle +2\psi|\nabla u|^2+\epsilon_{0}\beta \psi\big)e^{2u}\\
\geq & \, \epsilon_{0}F^{ij}u_{i}u_{j}+F^{ij}U_{ij}-C'
\end{split}
\end{equation}
for some uniform constant $C'>0$, where $U=\{U_{ij}\}$ is a symmetric matrix defined by
\[
\begin{split}
U_{ij}=&\Big(\frac{2-2t}{n-2} +\frac{2-t}{2} \Big)\epsilon_{0}|\nabla u|^2\delta_{ij}+\frac{1}{\beta}\frac{1-t}{n-2}R_{mlmk}u_{k}u_{l}\delta_{ij}\\
&+\frac{1}{\beta}u_{l}\nabla_{l}A^{t}_{ij} +\frac{1}{\beta}R_{iljk}u_{k}u_{l}+\phi'A^{t}_{ij}.
\end{split}
\]
Let $\lambda_{1}(U)\geq \cdots \geq \lambda_{n}(U)$ be the eigenvalues of $U$. We divide the proof into two cases:
\begin{itemize}
\item[(i)] If $\lambda_{n}(U)\geq 1$, then \eqref{inserting 3} implies
\begin{equation}\label{gradient}
C'\geq \epsilon_{0}F^{ij}u_{i}u_{j}+\sum_{i=1}^{n}F^{ii}\,,
\end{equation}
so that
\begin{equation*}
F^{ii}\leq C' \qquad \textrm{for all } i=1,\cdots, n\,.
\end{equation*}
Making use of the Appendix Lemma \ref{det}  below, we have
\begin{equation*}
F^{ii}\geq C(n,p)C'^{1-n}  \qquad \textrm{for all } i=1,\cdots, n\,.  
\end{equation*}
Substituting this into \eqref{gradient} yields the desired bound for  $|\nabla u|(x_0)$, thus completing the proof.
\item[(ii)] Otherwise, we have $\lambda_{n}(U)\leq 1$. There exists an unit eigenvector $\xi_{n}\in T_{x_0}M$ of $U$ such that 
\[
1\geq \lambda_{n}(U)=\langle U \xi_{n},\xi_{n}\rangle \geq \Big(\frac{2-2t}{n-2} +\frac{2-t}{2} \Big)\epsilon_{0}|\nabla u|^2-\frac{C'}{\beta}|\nabla u|-C''\,,
\]
where $C'$ and $C''$ are two uniform positive constants. This again yields the gradient estimate.
\end{itemize}
Therefore, the proof is complete.
\end{proof}

\subsection{Second order estimate}\label{Second order estimate}
In this subsection, we will prove the following global second order estimate:
\begin{proposition}\label{int hessian estimate}
Assume  $\lambda(-g^{-1}A_{g}^{t})\in \Gamma$ for some $t<1$. Let $u\in C^{4}(M)$ be a solution to  \eqref{scalar curvature problem}. Then we have
\begin{equation*} 
\sup_{M}|\nabla^2u|\leq C_{2}\,,
\end{equation*}
where $C_{2}$ is a constant depends on $p,$  $C_{0}$, $C_{1}$, and  $ \psi$.
\end{proposition}

\begin{proof}
For each $x\in M$, let $S(T_{x}M)$ denote the unit tangent bundle of $M$ at $x$. For a large constant $B$ to be specified later, we define the quantity
\begin{equation*}
G_{0}=\max_{x\in M}\max_{\xi\in S(T_{x}M)}\bigg\{\nabla^2 u(\xi,\xi)+\frac{B}{2}|\nabla u|^{2}\bigg\}\,.
\end{equation*}
Our goal is to estimate $G_0$. Suppose $G_0$ is attained at a point $x_0 \in M$ and a unit vector field $\xi_0 \in S(T_{x_0} M)$ such that
\[
G_{0}=\nabla^2 u(\xi_{0},\xi_{0})+\frac{B}{2}|\nabla u|^{2}(x_{0})\,.
\]
In a neighborhood of $x_0$, we may choose a smooth orthonormal frame $\{e_1, \cdots, e_n\}$ satisfying:
\begin{itemize}\itemsep0.5em
\item[(i)] $e_{1}(x_0)=\xi_{0}$;
\item[(ii)] $\nabla_{e_i}e_j=0$ at $x_0$, i.e. $\Gamma_{ij}^k(x_0)=0$ for all $i,j$ and $k$;
\item[(iii)] the symmetric matrix $\{\nabla^2u(e_{i},e_{j})\}$ is diagonal at $x_0$.
\end{itemize}

Now consider the auxiliary function
\[
G=\nabla^2u(e_1,e_1)+\frac{B}{2}|\nabla u|^2=u_{11}-\Gamma_{11}^{l}u_{l}+\frac{B}{2}|\nabla u|^2\,.
\]
By the definition of $G_0$, the function $G$ attains its maximum $G_0$ at $x_0$.
The maximum principle  at $x_0$ then implies that
\begin{equation}\label{vanishes 1}
\begin{split}
0=\nabla_{i}G &= u_{11i}-\Gamma_{11}^{l}u_{li}-(e_{i}\Gamma_{11}^{l})u_{l}+\frac{B}{2}(e_{i}g^{kl})u_{k}u_{l}+Bg^{kl}(e_{i}u_{k})u_{l}\\
&=u_{11i}-(e_{i}\Gamma_{11}^{l})u_{l}+Bu_{ik}u_{k}\,,
\end{split}
\end{equation}
and
\begin{equation}\label{max principle}
\begin{split}
0\geq \overline{F}^{ij} u_{11ij}&-2\overline{F}^{ij}(e_{i}\Gamma_{11}^{j})u_{jj}-\overline{F}^{ij}(e_{j}e_{i}\Gamma_{11}^{l})u_{l}+\frac{B}{2}\overline{F}^{ij}e_{j}e_{i}(g^{kl})u_{k}u_{l}\\&+B\overline{F}^{ij}u_{ijl}u_{l}+B\overline{F}^{ii}u_{ii}^{2}\,,
\end{split}
\end{equation}
where ${\overline{F}^{ij}}$ is defined as in \eqref{relation}.

We first proceed to compute the last two terms in \eqref{max principle}. Differentiating the equation \eqref{scalar curvature problem} in the direction $e_l$ and using the fact that $\Gamma_{ij}^{k}=0$ at $x_0$, we deduce from \eqref{vanishes1} that
\begin{equation}\label{mix term}
\begin{split}
(\psi_{l}+2\psi u_{l})e^{2u}&=F^{ij}\big[(W[u])_{ij}\big]_{l}\\
&=F^{ij}\big(u_{ijl}-u_{k}(\Gamma_{ij}^{m})_{l} \big)-F^{ij}\big(u_{il}u_{j}+u_{i}u_{jl}+\nabla_{l}A^{t}_{ij}\big)\\
& \quad +\Big(\frac{1-t}{n-2}(u_{kkl}-u_{i}(\Gamma_{kk}^{i})_{l})+(2-t)u_{k} u_{kl}\Big)\mathcal{F}\\
&=\overline{F}^{ij}u_{ijl}-F^{ij}\big(2u_{il}u_{j}+\nabla_{l}A^{t}_{ij}+u_{k}(\Gamma_{ij}^{m})_{l}\big)\\
&\quad +\big((2-t)u_{k} u_{kl}-\frac{1-t}{n-2}u_{i}(\Gamma_{kk}^{i})_{l}\big)\mathcal{F}\,.
\end{split}
\end{equation}
Let us denote $W_{ij}=(W[u])_{ij}$ for simplicity. Then
\[
W_{ij}=u_{ij}+\frac{1-t}{n-2}(\Delta u)\delta_{ij}+\frac{2-t}{2}|\nabla u|^2\delta_{ij}-u_{i}u_{j}-A^{t}_{ij}\,.
\]
Since $W[u] \in \Gamma$ and by gradient estimates, we have
\[
0\leq \sum_{i=1}^{n}W_{ii}\leq \big(1+\frac{n(1-t)}{n-2}\big)\big(\Delta u-nC'\big)
\]
for some uniform positive constant $C'$. Therefore, we obtain
\[
\sum_{i=1}^{n}(u_{ii}-C')\geq 0\,.
\]
Note that $u_{11}$ is the largest eigenvalue of $\nabla^2u$, so for all $i=1,2,\cdots,n$,
\[
\begin{split}
|u_{ii}|&=|u_{ii}-C'+C'|\leq |u_{ii}-C'|+C'\\
&\leq  (n-1)(u_{11}-C')+C'=(n-1)u_{11}+nC'\,.
\end{split}
\]
It follows from \eqref{mix term} and the gradient estimates that
\begin{equation}
B\overline{F}^{ij}u_{ijl}u_{l}\geq -BC' (1+\mathcal{F}+u_{11}\mathcal{F} )\,.
\end{equation}

For the last term in \eqref{max principle}, we obtain
\begin{equation}\label{square}
\overline{F}^{ii}u_{ii}^{2}=\sum_{i=1}^{n}\Big( F^{ii}+\frac{1-t}{n-2}\mathcal{F}\Big)u_{ii}^{2}\geq \frac{1-t}{n-2}\mathcal{F} u_{11}^{2}
\end{equation}
For the second to fourth terms in \eqref{max principle}, we have
\begin{equation}\label{234}
-2\overline{F}^{ij}(e_{i}\Gamma_{11}^{j})u_{jj}-\overline{F}^{ij}(e_{j}e_{i}\Gamma_{11}^{l})u_{l} +\frac{B}{2}\overline{F}^{ij}e_{j}e_{i}(g^{kl})u_{k}u_{l}\geq - C'(1+u_{11}\mathcal{F}+B \mathcal{F})\,.
\end{equation}
Substituting \eqref{mix term}, \eqref{square}, and \eqref{234} into \eqref{max principle}, we deduce that
\begin{equation}\label{max principle v2}
\begin{split}
BC' (1+\mathcal{F}+u_{11}\mathcal{F} )\geq &\quad \overline{F}^{ij} u_{11ij}+\frac{B(1-t)}{n-2}\mathcal{F} u_{11}^{2}\,.
\end{split}
\end{equation}

Differentiating the equation \eqref{scalar curvature problem} in the $e_1$-direction twice, we obtain
\begin{equation*}
F^{ij,kl}e_{1}(W_{ij})e_{1}(W_{kl})+F^{ij}e_{1}e_{1}(W_{ij})=(\psi_{11}+4\psi_{1}u_{1}+4fu_{1}^2+2\psi u_{11})e^{2u}\,.
\end{equation*}
Since $F$ is concave on $\Gamma$, the first term on the left-hand side must be non-positive. Therefore, we have
\begin{equation}\label{fourth term}
F^{ij}e_{1}e_{1}(W_{ij})\geq -C'(1+u_{11})\,.
\end{equation}
A direct computation at $x_0$ gives
\[
\begin{split}
e_{1}e_{1}(W_{ij})=&-e_{1}e_{1}(A^{t}_{ij})+u_{ij11}+2u_{k1}e_{1}(\Gamma_{ij}^{k})+u_{k}e_{1}e_{1}(\Gamma_{ij}^{k})  \\
&+\frac{1-t}{n-2}\big(u_{kk11}g_{ij}+u_{kk}e_{1}e_{1}(g_{ij})-2u_{l1}e_{1}(\Gamma_{kk}^{l})+u_{r}e_{1}e_{1}(\Gamma_{kk}^{l}) \big)\\
&-(u_{i11}u_{j}+u_{i}u_{j11}+2u_{i1}u_{j1})+\frac{2-t}{2}e_{1}e_{1}(g^{kl})u_{k}u_{l}\delta_{ij}\\
&+(2-t)\big(u_{k11}u_{k}+u_{11}^{2}\delta_{ij}+|\nabla u|^2e_{1}e_{1}(g_{ij}) \big)\,.
\end{split}
\]
Substituting this into \eqref{fourth term} and using \eqref{vanishes 1} to replace terms involving $u_{11i}$, we obtain
\begin{equation}\label{fourth term v2}
-C'(1+u_{11})\leq \overline{F}^{ij} u_{11ij}+C'(1+u_{11})\mathcal{F}+(2-t)\mathcal{F}u_{11}^{2}\,.
\end{equation}
Combining \eqref{max principle v2} and \eqref{fourth term v2}, we deduce that
\begin{equation*}
\begin{split}
 \Big(\frac{B(1-t)}{n-2}-(2-t)\Big)\mathcal{F} u_{11}^{2}\leq &BC' (1+\mathcal{F}+u_{11}\mathcal{F} )+C'(1+u_{11})+C'(1+u_{11})\mathcal{F}\\
 =& (1+B)C'(1+u_{11})\mathcal{F}+C'(1+u_{11}+B)\,.
\end{split}
\end{equation*}

Since we have assumed $t<1$,  we may now choose $B$ to be sufficiently large such that
\[
\frac{B(1-t)}{n-2}-(2-t)=2.
\]
Then it follows that
\[
\big(2u_{11}^{2}- (1+B)C'(1+u_{11}) \big)\mathcal{F}\leq C'(1+u_{11}+B)\,.
\]
Without loss of generality, we may assume $u_{11}$ is sufficiently large so that 
\[
u_{11}^{2}- (1+B)C'(1+u_{11})\geq 0\,,
\]
since otherwise the desired bound already holds. Therefore, we infer that
\[
C'(1+u_{11}+B)\geq \mathcal{F}u_{11}^{2}\geq \mathscr{T}_{0}u_{11}^{2},
\]
where we have used the fact that $\mathcal{F} \geq \mathscr{T}_{0}$ for some positive constant $\mathscr{T}_{0}$. Indeed, by the Cauchy-Schwarz inequality, we have
\[
\mathcal{F}=\sum_{i=1}^{n}F^{ii}\geq n\Big(\prod_{i=1}^{f}F^{ii} \Big)^{1/n}\geq n\mathscr{T}^{1/n}=:\mathscr{T}_{0}.
\]
This yields the desired upper bound for $u_{11}(x_0)$, and hence an upper bound for $G_0$. The proof is then complete.
\end{proof}

\section{Proof of Theorem \ref{main}}\label{proof}
With the $C^2$ estimates from Proposition \ref{int hessian estimate} established, the proof of Theorem \ref{main} follows by a standard continuity argument, as implemented in \cite{GV03}.

Fix a parameter $s \in [0,1]$ and consider the family of equations
\begin{equation}\label{family}
F\Big(\nabla^2u_{s}+\frac{1-t}{n-2}(\Delta u_{s})g+\frac{2-t}{2}|\nabla u_{s}|^2g-du_{s}\otimes du_{s}-\tilde{A}^{t}_{g}\Big)=\psi_{s}e^{2u_{s}}\,.
\end{equation}
Here, $\psi_s = s f + (1-s)$ and $\tilde{A}^{t}_{g} = s A^t_g - (1-s)\lambda g$, where $\lambda$ is a positive constant chosen so that
\[
\lambda F(I_{n})=1\,.
\]
Define the set
\[
S=\{s\in [0,1]\,:\, \eqref{family} \textrm{ admits a solution } u_{s}\in C^{2,\alpha}(M)\}\,.
\]
To apply the continuity method, we verify that $S$ is non-empty open and closed.

Note that when $s = 0$, the constant function $u_0 \equiv 0$ is a solution to \eqref{family}, so $0 \in S$ and $S \neq \emptyset$.  By Lemma \ref{elliptic}, the linearization $L: C^{2,\alpha}(M) \to C^{\alpha}(M)$ with
\[
\begin{split}
L(v)=\overline{F}^{ij}(W[u_{s}])\nabla_{i}\nabla_{j}v+F^{ij}\big(W[u_{s}]\big)\chi_{ij}(dv)-\psi e^{2u_{s}}v
\end{split}
\]
is an elliptic operator. The ellipticity of $L$ ensures its invertibility, which implies that $S$ is open. 

We now establish that $S$ is closed by invoking the a priori estimates from Proposition \ref{int hessian estimate}. Let $\{s_{i}\}\subset S$ be a sequence converging to some $s_{\infty}\in [0,1]$.
Under the hypothesis that $\lambda(-g^{-1}A_{g}^{t})\in \Gamma$, we also have $\lambda(-g^{-1}\tilde{A}^{t}_{g})\in \Gamma$. Therefore, the $C^2$ estimates of $u_{s}$--applicable upon replacing $\psi$ by $\psi_{s}$--persist in this setting.  To upgrade these to $C^{2,\alpha}$ estimates, we can now invoke the Evans-Krylov theorem \cite{E82,GT83}, yielding
\[
\|u_{s_i}\|_{C^{2,\alpha}(M)}\leq C\,,
\]
where $C$ is a constant independent of $i$. Differentiating the equation \eqref{family} now gives a uniformly elliptic PDE with uniformly H\"older coefficients. Applying the Schauder estimates then gives
\[
\|u_{s_i}\|_{C^{4,\alpha}(M)}\leq C.
\]
By the Arzel\`a-Ascoli theorem, a subsequence of $\{u_{s_i}\}$  converges in $ C^{4,\alpha}(M)$ to a limit function $u_{\infty}$, which solves the equation \eqref{family} at $s_{\infty}$. Hence, $s_{\infty}\in S$, and $S$ is closed. Since $S$ is nonempty, open and closed, it follows that  $S=[0,1]$.

This establishes the existence of a $C^{4,\alpha}$ solution to Equation \eqref{scalar curvature problem}. By repeatedly differentiating the equation and applying Schauder theory, we conclude that the solution is in fact smooth. Uniqueness follows from the maximum principle. This completes the proof of Theorem \ref{main}.

\section{Prescribed $\mathcal{M}_p$-curvature problem}
In this section, we will consider a so-called prescribed $\mathcal{M}_p$-curvature problem on closed Riemannian manifolds with negative curvature. In the context of $p$-geometry and $p$-potential theory, Harvey and Lawson  \cite{HL13} introduced the following $p$-convex cone
\[
\mathcal{P}_{p}:=\{ \lambda\in \R^n \,: \,\lambda_{i_1}+\cdots +\lambda_{i_{p}}>0,\, 1\leq i_1<\cdots< i_{p}\leq n\}\,,
\]
which is clearly connected for a fixed integer for a fixed integer $1\leq p\leq n$. It is worth noting that $\mathcal{P}_{1}=\Gamma_{n}$ and $\mathcal{P}_{n}=\Gamma_{1}$.

Given an open convex symmetric cone $\Gamma\subset\mathbb{R}^n$ with vertex at the origin satisfying $\Gamma_{n}\subset \Gamma \subset \Gamma_{1}$, we define the associated nonnegative constant $\mu_{\Gamma}^{+}\in [0,n-1]$  as the unique number characterized by
\[
 (-\mu^{+}_{\Gamma}, 1, \cdots,1)\in \partial\Gamma.
 \]
This constant was introduced by Li and Nguyen \cite{LN14b} in their work on such cones. The following observation is well-known.
\begin{lemma}  For any integer $p= 1,2, \cdots,n$, we have $\mu^{+}_{\mathcal{P}_{p}}=p-1.$  
\end{lemma}

Let us consider the $p$-fold sum operator $\mathcal{M}_{p}:\mathcal{P}_{p}\rightarrow \mathbb{R}$ defined by
\begin{equation}
\mathcal{M}_{p}(\lambda) = \prod_{1 \leq i_1 < i_2 < \cdots < i_{p} \leq n} (\lambda_{i_1} + \lambda_{i_2} + \cdots + \lambda_{i_p}) \,.
\end{equation}
This operator is symmetric in the eigenvalues. In contrast to the elementary symmetric functions $\sigma_k$,
expressing $\mathcal{M}_p$ in terms of elementary symmetric polynomials is nontrivial; further details can be found in \cite{D21}. A straightforward computation shows that 
$\mathcal{M}^{1/\binom{n}{p}}_{p}$ is a concave elliptic operator on $\mathcal{P}_{p}$. Moreover,  it was proved in \cite{ZZ25} that $(\mathcal{M}^{1/\binom{n}{p}}_{p},\mathcal{P}_{p})$ satisfies Condtion T. Hence, by Theorem \ref{GD}, we obtain the following existence result.
\begin{proposition}\label{p curvature problem}
Let $(M,g)$ be a closed Riemannian manifold with $\lambda(-g^{-1}A_{g}^{t})\in \mathcal{P}_{p}$ for some $t<1$, and let $\psi$ be a smooth positive function on $M$. Then there exists a unique conformal metric $\tilde{g}=[g]$ satisfying
\begin{equation}
\mathcal{M}^{1/\binom{n}{p}}_{p}(\lambda(-\tilde{g}^{-1}A_{\tilde{g}}^{t}))=\psi\,, \qquad \lambda(-\tilde{g}^{-1}A_{\tilde{g}}^{t})\in \mathcal{P}_{p}.
\end{equation}
In particular, if $\lambda(-g^{-1}\mathrm{Ric}_{g})\in \mathcal{P}_{p}$, then there exists a unique conformal metric $\tilde{g}\in [g]$ such that 
\[
\mathcal{M}^{1/\binom{n}{p}}_{p}(\lambda(-\tilde{g}^{-1}\mathrm{Ric}_{\tilde{g}})) =constant\,, \qquad \lambda(-\tilde{g}^{-1}\mathrm{Ric}_{\tilde{g}})\in \mathcal{P}_{p}. 
\]
\end{proposition}

\begin{remark}
By the definition of $\mathcal{P}_{p}$, the assumption that $\lambda(-g^{-1}\mathrm{Ric}_{g}) \in \mathcal{P}_{p}$ implies that the sum of the $p$ largest eigenvalues is negative. Hence, it represents an intermediate curvature condition that lies strictly between negative scalar curvature and negative Ricci curvature. We anticipate that this notion will attract further attention in future research.
\end{remark}

\end{sloppypar}	

\begin{thebibliography}{99}



 


\bibitem{BG08} Thomas Branson and  Ashwin Rod  Gover, 
Variational status of a class of fully nonlinear curvature
prescription problems. 
{\em Calc. Var. Partial Differential Equations.,}
32, no. 2, (2008), 253--62.

\bibitem{BV04} Simon Brendle and Jeff Viaclovsky, A variational characterization for $\sigma_{\frac{n}{2}}$.
{\em Calc. Var. Partial Diff. Equ.,}
20, no. 4 (2004), 399--402.



\bibitem{CNS85} Luis Angel Caffarelli,  Louis Nirenberg, and Joel Spruck, The Dirichlet problem for nonlinear second-order elliptic equations, III: Functions of the eigenvalues of the Hessian.  {\em Acta Math.,} 38 (2) (1985), 209--252.

\bibitem{CGY02a} Sun-Yung Alice Chang, Matthew Gursky, and Paul Yang,  An equation of Monge-Amp\`ere type in conformal
geometry, and four-manifolds of positive Ricci curvature.
{\em Ann. Math.,} 
155 (2),  (2002), 709--87.

\bibitem{CGY02b} Sun-Yung Alice Chang, Matthew Gursky, and Paul Yang, An a priori estimate for a fully nonlinear equation on
four-manifolds. 
{\em J. Anal. Math.,} 87, no. 2 (2002), 151--86.

\bibitem{DW25} Jonah A. J. Duncan and   Yi Wang,  Fully nonlinear Yamabe-type problems on non-compact manifolds. {\em Comm. Anal. Geom.,} 33 (2025), no. 5, 1101--1133.

\bibitem{DL25} Jonah A. J. Duncan and Luc Nguyen,  The $\sigma_k$-Loewner-Nirenberg problem on Riemannian manifolds for $k<\frac{n}{2}$. {\em Anal. PDE.,} 18 (2025), no. 9, 2203--2240.

\bibitem{DL26} Jonah A. J. Duncan and Luc Nguyen,  The $\sigma_k$-Loewner-Nirenberg problem on Riemannian manifolds for $k=\frac{n}{2}$ and beyond. {\em J. Funct. Anal.,} 290 (2026), no. 6, Paper No. 111306.


\bibitem{CH24} Li Chen and Yan He,  A class of fully nonlinear equations on Riemannian manifolds with negative curvature. 
{\em Calc. Var. Partial Differential Equations.,}
63 (2024), no. 6, Paper No. 158, 17 pp.

\bibitem{CGH22} Li Chen, Xi Guo, and Yan He, A class of fully nonlinear equations arising in conformal geometry. 
{\em Int. Math. Res. Not. IMRN.,} 
5, (2022), 3651--3676. 



\bibitem{CW}  Kai-Seng Chou and Xu-Jia Wang, A variational theory of the Hessian equation, {\em Comm. Pure Appl. Math.,} 54 (2001), no 9. 1029--1064. 

\bibitem{Dong06}  Hongjie Dong, Hessian equations with elementary symmetric functions. {\em Commun. Partial Differ. Equ.,} 31 (2006), 1005--1025.


\bibitem{D21} Slawomir Dinew, Interior estimates for $p$-plurisubharmonic functions.  {\em Indiana Univ. Math. J.,} 72 (2023), no. 5, 2025--2057.






 

\bibitem{E82}  Lawrence Craig Evans,  Classical solutions of fully nonlinear, convex, second-order elliptic equations.  {\em Comm. Pure Appl. Math.,} 35 (1982), no. 3, 333--363.

 


\bibitem{Guan94} Bo Guan,  The Dirichlet problem for a class of fully nonlinear elliptic equations, {\em Comm. Partial Differential Equations.,} 19  (1994), 399--416.

\bibitem{Guan99} Bo Guan, The Dirichlet problem for Hessian equations on Riemannian manifolds.  {\em Calc. Var. Partial Differential Equations.,}  8 (1999), no. 1, 45--69.

\bibitem{Guan08} Bo Guan, Complete conformal metrics of negative Ricci curvature on compact manifolds with boundary. {\em Int. Math. Res. Not. IMRN.,} (2008), Art. ID rnn 105, 25 pp.

\bibitem{Guan12} Bo Guan, Second order estimates and regularity for fully nonlinear elliptic equations on Riemannian manifolds. {\em Duke Math. J.,} 163 (2014), 1491--1524. 

\bibitem{GL96} Bo Guan and Yanyan Li, Monge-Amp\`ere equations on Riemannian manifolds. {\em J. Differential Equations.,} 132 (1996), no. 1, 126--139.

 

\bibitem{GW03} Pengfei Guan and Guofang Wang, A fully nonlinear conformal flow on local conformally flat manifolds.  {\em J. Reine Angew. Math.,} 557 (2003): 219--38.


\bibitem{GT83} David Gilbarg Gilbarg and Neil Trudinger, Elliptic partial differential equations of second order, second ed., 
{\em Springer-Verlag, Berlin,} (1983).

\bibitem{GV03} Matthew Gursky and Jeff Viaclovsky, 
Fully nonlinear equations on Riemannian manifolds with
negative curvature.  {\em Indiana Univ. Math. J.,} 52, no. 2 (2003), 399--419.

\bibitem{GV07} Matthew Gursky and Jeff Viaclovsky, Prescribing symmetric functions of the eigenvalues of the
Ricci tensor.  {\em Ann. Math.,} 166, no. 2 (2007), 475--531.

\bibitem{GW06} Yuxin Ge and Guofang Wang, On a fully nonlinear Yamabe problem.
{\em Ann. Sci. \'Ec Norm. Sup\'er.,}
39 (4), (2006), 569--98.

\bibitem{GP24} Bin Guo and Duong Hong  Phong, 
On $L^{\infty}$ estimates for fully nonlinear partial differential equations, {\em Ann. of Math.,} (2) 200 (2024), no. 1, 365--398.

\bibitem{HL09} F. Reese Harvey and H. Blaine Lawson Jr, Dirichlet duality and the nonlinear Dirichlet problem. {\em Commun. Pure Appl. Math.,} 62 (2009), 396--443.

\bibitem{HL11} F. Reese Harvey and H. Blaine Lawson Jr, Dirichlet duality and the nonlinear Dirichlet problem on Riemannian manifolds. {\em J. Differ. Geom.,} 88 (2011), 395--482.

\bibitem{HL13} F. Reese Harvey and H. Blaine Lawson Jr, $p$-convexity, $p$-plurisubharmonic and the Levi problem.  {\em Indiana Univ. Math. J.,} 62 (2013), no. 1, 149--169.

\bibitem{HL23} F. Reese Harvey and H. Blaine Lawson Jr, Determinant majorization and the work of Guo-Phong-Tong and Abja-Olive, {\em Calc. Var. Partial Differential Equations.,} 62 (2023), no. 5, Paper No. 153, 28 pp.


 
\bibitem{Iv91} Nina Ivochkina, Solution of the Dirichlet problem for equations of mth order curvature. {\em Engl. Trans. Leningr. Math. J.,} 2 (1991), no. 3, 631--654.

\bibitem{ITW04} Nina Ivochkina, Neil Trudinger, and  Xu-Jia Wang, The Dirichlet problem for degenerate Hessian equations. {\em Commun. Partial Differ. Equ.,} 29 (2004), 219--235. 



\bibitem{Kry83} Nicolai Krylov,  On degenerate nonlinear elliptic equations, II. {\em Mat. Sb.,} 121 (163) (1983) 211--232, English translation: {\em Math. USSR Sb.,} 49 (1984) 207--228.

\bibitem{Kry95} Nicolai Krylov, On the general notion of fully nonlinear second order elliptic equations. {\em Trans. Am. Math. Soc.,} 347 (1995), 857--895.

\bibitem{Kry97} Nicolai Krylov, Fully nonlinear second order elliptic equations: recent development, Dedicated to Ennio De Giorgi. {\em Ann. Sc. Norm. Super. Pisa, Cl. Sci.,} (4) 25, (1997) 569--595. 


\bibitem{LS05} Jiayu Li and Weiming Sheng,  Deforming metrics with negative curvature by a fully nonlinear flow. {\em Calc. Var. Part. Differ. Equ.,} 23, (2005), 33--50.

\bibitem{LL03} Aobing Li and  Yanyan Li,  On some conformally invariant fully nonlinear equations. {\em Comm. Pure Appl. Math.,}
56, (2003), 1416--64.

\bibitem{LL03b} Aobing Li and Yanyan Li,  On some conformally invariant fully nonlinear equations. II. Liouville, Harnack and Yamabe. {\em Acta Math.,}  195, (2005), 117--154. 

 
\bibitem{Li90} Yanyan Li, Some existence results for fully nonlinear elliptic equations of Monge-Amp\`ere type. 
{\em Comm. Pure Appl. Math.,} 43,  (1990), 233--271.

\bibitem{LN14} Yanyan Li and  Luc Nguyen,  A compactness theorem for a fully nonlinear Yamabe problem under a lower Ricci curvature bound.  
{\em J. Funct. Anal.,}
266, no. 6, (2014), 3741--71.

\bibitem{LN14b} Yanyan Li and  Luc Nguyen, Harnack inequalities and B\^ocher-type theorems for conformally invariant, fully nonlinear degenerate elliptic equations. {\em Comm. Pure Appl. Math.,} 67 (2014), no. 11, 1843--1876.



\bibitem{STW13} Weiming Sheng,  Neil Trudinger, and Xu-Jia Wang,  The Yamabe problem for higher order
curvatures. 
{\em J. Differential Geom.,}
77, (2007), 515--53.
 


 

\bibitem{Tso} Kaising Tso, On a class of Monge-Amp\`ere functional. {\em Invent. Math.,} 101 (1990) 425--448.

\bibitem{TW99} Neil Trudinger and  Xu-Jia Wang, Hessian measures. II. {\em Ann. Math.,} 2 (1999), no. 150, 579--604.

\bibitem{Tru95} Neil Trudinger,  On the Dirichlet problem for Hessian equations. {\em Acta Math.,} 175 (1995), 151--164.  

\bibitem{Urbas} John Urbas, Hessian equations on compact Riemannian manifolds, in: Nonlinear Problems in Mathematical Physics and Related Topics II, Kluwer/Plenum, New York, (2002), pp. 367--377.

\bibitem{Via00a} Jeff Viaclovsky,
Conformal geometry, contact geometry, and the calculus of variations. 
{\em Duke Math. J.,}
101, (2000), 283--316 


\bibitem{Via00b} Jeff Viaclovsky, Conformally invariant Monge-Amp\`ere equations: global solutions. 
{\em Trans. Am. Math. Soc.,}
352(9), (2000), 4371--4379.
 

\bibitem{Via02} Jeff Viaclovsky,   Estimates and existence results for some fully nonlinear elliptic equations on Riemanian manifolds. 
{\em Commun. Anal. Geom.,}
10(4), (2002), 815--846. 

\bibitem{W94} Xu-Jia Wang, A class of fully nonlinear elliptic equations and related functionals. {\em  Indiana Univ. Math. J.,} 43 (1), (1994), 25--54.


\bibitem{Yuan01} Yu Yuan, A priori estimates for solutions of fully nonlinear special Lagrangian equations. {\em Ann. Inst. Henri Poincar\'e, Anal. Non Lin\'aire.,} 18 (2001), 261--270.

\bibitem{ZZ25} Jiaogen Zhang and Tao Zheng, The Dirichlet eigenvalue problem for $p$-Monge-Amp\`ere operator. Accepted by {\em  Indiana Univ. Math. J.}

\bibitem{Zhang25} Jiaogen Zhang, The Dirichlet eigenvalue problems for some concave elliptic Hessian operators. {\em arXiv:} 2510.16748.

\end{thebibliography}
\end{document}